\documentclass[12 pt]{amsart}

\usepackage{amsfonts, amsmath, amssymb, amsthm, amsopn, latexsym, hyperref, graphicx, mathrsfs, verbatim, enumerate, stmaryrd} 
\usepackage{subfig} 
\usepackage[all]{hypcap} 
\usepackage{pst-all, multido} 
\usepackage[margin=1in]{geometry} 
\usepackage{Style} 

\begin{document}

\setcounter{tocdepth}{1}

\renewcommand{\qedsymbol}{$\square$}


\newtheorem*{equid}{Theorem}
\newtheorem*{conj}{Conjecture}
\newtheorem*{mainthm}{Main Theorem}
\newtheorem*{nonumprop}{Proposition}
\newtheorem{letthm}{Theorem}
\renewcommand*{\theletthm}{\Alph{letthm}}
\newtheorem{thm}{Theorem}[section]
\newtheorem{lem}[thm]{Lemma}
\newtheorem{cor}[thm]{Corollary}
\newtheorem{prop}[thm]{Proposition}
\theoremstyle{definition}
\newtheorem{rem}[thm]{Remark}
\newtheorem*{War}{Warning}
\newtheorem*{Quest}{Questions}
\newtheorem{Def}[thm]{Definition}
\newtheorem{Not}[thm]{Notation}
\newtheorem{Ass}[thm]{Assumption}
\newtheorem{ex}[thm]{Example}
\newtheorem*{exs}{Examples}
\newtheorem{obs}[thm]{Observation}
\newtheorem*{Ack}{Acknowledgements}

\numberwithin{equation}{section}


\title[On the growth of local intersection multiplicities in holomorphic dynamics]{On the growth  of local intersection multiplicities in holomorphic dynamics: a conjecture of Arnold}
\author[W. Gignac]{William Gignac}
\address{Centre de mathématiques Laurent Schwartz, \'{E}cole polytechnique, 91128 Palaiseau cedex, France}
\email{william.gignac@math.polytechnique.fr}
\date{\today}

\begin{abstract} We show by explicit example that local intersection multiplicities in holomorphic dynamical systems can grow arbitrarily fast, answering a question of V.\ I.\ Arnold. On the other hand, we provide results showing that such behavior is exceptional, and that typically local intersection multiplicities grow subexponentially.
\end{abstract}

\maketitle


\section*{Introduction}

Let $f\colon(\C^2,0)\to (\C^2,0)$ be a germ of a holomorphic map that fixes the origin $0\in \C^2$, and which is finite-to-one near $0$. Suppose that $C$ and $D$ are two germs of holomorphic curves passing through $0$. In this article, we will study the sequence $\mu(n) := C\cdot f^n(D)$ of local intersection multiplicities  at the origin, for $n\geq 0$. Specifically, we will address the question: \emph{how fast can the sequence $\mu(n)$ grow?} This and related questions were posed and studied by V.\ ~I.\ ~Arnold, who conjectured that, if $\mu(n)<\infty$ for every $n$, the sequence $\mu(n)$ grows at most exponentially fast, see \cite[\S5]{MR1215971}, \cite[p.\ 215]{MR1350971}, and \cite[problems 1994-49 and 1994-50]{MR2078115}. Arnold proved this conjecture in the case when $f$ is a local biholomorphism and in some cases when the complex derivative $f'(0)$ has exactly one zero eigenvalue \cite[\S5]{MR1215971}, but the general case appears to be unknown. A new proof of the conjecture in the case when $f$ is  a local  biholomorphism has recently been obtained by Seigal and Yakovenko \cite{SY}.

In this article, we will show by explicit construction that Arnold's conjecture is false in general, and that in fact the sequence $\mu(n)$ can grow arbitrarily fast. More precisely, we will prove the following theorem.

\begin{letthm}\label{thmA} Let $f\colon \C^2\to \C^2$ be the polynomial map $f(x,y) = (x^2 - y^4,y^4)$ and $\nu\colon \N\to \R$ be any function. Then there exist germs of holomorphic curves $C$ and $D$ through the origin such that the local intersection multiplicities $\mu(n) = C\cdot f^n(D)$ are always finite, and such that  $\mu(n)>\nu(n)$ for infinitely many $n$.
\end{letthm}

Note that the complex derivative $f'(0) $ for this map is $0$, so $f$ defines a \emph{superattracting} germ at $0$ (in general, $f$ is superattracting if $f'(0)$ is nilpotent). The dynamics of superattracting germs is an active area of research in holomorphic dynamics in several variables, see for instance \cite{MR1275463, MR1759437, MR2339287, MR2904007, Rug2, MR2853790, BEK, GR} and the notes \cite{Mattias}.

It has long been known that intersection multiplicities in \emph{smooth} dynamical systems (local or global) can grow arbitrarily fast \cite{MR1039340, MR1167716, MR1199715}. On the other hand,  it is also known that intersection multiplicities ``generically" grow at most exponentially (see \cite{MR1139553} for a precise formulation). Our second theorem is a holomorphic version of this principle.

\begin{letthm}\label{thmB} Let $f\colon (\C^d, 0)\to (\C^d,0)$ be a holomorphic fixed point germ at the origin $0\in \C^d$, where $d\geq 2$, such that $f$ is finite-to-one near $0$. Fix holomorphic function germs $\psi_1,\ldots, \psi_m$ at the origin such that $\{\psi_1 = 0\}\cap \cdots\cap \{\psi_m = 0\} = \{0\}$. For each $z\in \C^m$, let $D_z$ denote the hypersurface germ $D_z = \{z_1\psi_1 + \cdots + z_m\psi_m = 0\}$ through the origin. Fix an integer $k\in \{1,\ldots, d-1\}$. For each $a = (a_1,\ldots, a_k)\in (\C^m)^k$ and $b = (b_1,\ldots, b_{d-k})\in (\C^m)^{d-k}$, let $V_a$ and $W_b$ be the local intersection cycles $V_a :=  D_{a_1}\cdot\ldots \cdot D_{a_k}$ and $W_b := D_{b_1}\cdot\ldots\cdot D_{b_{d-k}}$. Then there exists a dense set $U\subseteq (\C^m)^k\times (\C^m)^{d-k}$, given as the complement of a countable union of algebraic subsets, with the property that for all $(a,b)\in U$, the cycles $V_a$ and $W_b$ are of codimension $k$ and $d-k$, respectively, and the sequence of local intersection multiplicities $\mu(n) := V_a\cdot f^n_*W_b$ grows subexponentially, that is to say, $\mu(n)\leq AB^n$ for some $A,B>0$.
\end{letthm}

In the special case when $d = 2$ and $f$ is superattracting, we will in fact be able to prove much more about the sequence $\mu(n)$.

\begin{letthm}\label{thmD} With the same setup and notations as \hyperref[thmB]{Theorem~\ref*{thmB}} but with the additional assumptions that $d = 2$ and $f$ is superattracting, we have for each $(a,b)\in U$ that the sequence of local intersection multiplicities $\mu(n) := V_a\cdot f^n_*W_b$ at the origin eventually satisfies an integral linear recursion relation. Moreover, there exist constants $A_1, A_2>0$ such that $A_1c_\infty^n \leq \mu(n)\leq A_2 c_\infty^n$ as $n\to \infty$, where here $c_\infty>1$ denotes the asymptotic attraction rate of $f$ (see \S5 for the definition of $c_\infty$). If one replaces $f$ by $f^2$, then in fact there is a constant $A>0$ such that $\mu(n)\sim Ac_\infty^n$.
\end{letthm}

The proof of \hyperref[thmB]{Theorem~\ref*{thmB}}, which will be given in \S4, is a rather easy application of Teissier's theory of mixed multiplicities. \hyperref[thmD]{Theorem~\ref*{thmD}}, which we will prove in \S5, relies on recent non-elementary results of the author and M.\ Ruggiero \cite{GR} within the subject of dynamics on valuation spaces. Unlike Theorems \ref{thmB} and \ref{thmD}, our proof of \hyperref[thmA]{Theorem~\ref*{thmA}} requires no high-powered techniques, and lends itself to an easy overview, which we give now.


Let $S$ be the space of binary sequences $S = \{0,1\}^\N$, and let $\sigma\colon S\to S$ denote the left-shift map on $S$. For any two sequences $s,t\in S$, set $M(s,t)$ to be the smallest index $m$ such that $s_m\neq t_m$, with $M(s,t) = \infty$ if $s = t$. To prove \hyperref[thmA]{Theorem~\ref*{thmA}}, we will construct a family $\{C_s\}_{s\in S}$ of holomorphic curve germs through the origin with the properties that \begin{enumerate}
\item[1.] $f(C_s) = C_{\sigma(s)}$ for each $s\in S$, and
\item[2.] for any $s,t\in S$, the local intersection multiplicity $C_s\cdot C_t$ is $\asymp 4^{M(s,t)}$.
\end{enumerate} The theorem then follows easily from the following simple proposition, the proof of which is left to the reader.

\begin{nonumprop} Let $\nu\colon \N\to \R$ be any function. Then there exist sequences $s,t\in S$ such that $M(s,\sigma^n(t))$ is finite for all $n\geq 0$, and such that $M(s,\sigma^n(t))>\nu(n)$ for infinitely many $n$.
\end{nonumprop}

In \S1, we will construct the $C_s$ as \emph{formal} curves, that is, as curves defined by irreducible formal power series $\varphi_s\in \C\llbracket x, y\rrbracket$. The coefficients of the power series $\varphi_s$ will be determined via a recursive procedure that guarantees properties 1.\ and 2.\ are satisfied. In \S2, we will prove that each formal power series $\varphi_s$ is actually convergent, and hence that the curve germs $C_s$ are \emph{holomorphic}. It should be noted that the construction of the power series $\varphi_s$ in \S1 is purely algebraic, and that if we replace the word \emph{holomorphic} with \emph{formal}, \hyperref[thmA]{Theorem~\ref*{thmA}} holds when $\C$ is replaced by any field of characteristic $\neq 2$. 

In \S3, we will sketch an alternative, geometric construction of the curves $C_s$, valid when working over $\C$. The construction realizes the curves as a \emph{Cantor bouquet} of holomorphic stable manifolds,  similar to the construction carried out first in \cite{MR1808626} for rational maps; see also the related works \cite{Yam2, MR2195140, MR2307152, Tomoko, MR2629648}. Finally, it is worth mentioning that this counterexample is by no means isolated; one can construct similar Cantor bouquets for many other superattracting germs.

\subsection*{Acknowledgements} I would like to wholeheartedly thank Mattias Jonsson and Charles Favre for their support and guidance during the course of this project. I would also like to thank the referee for very useful commentary, and especially for pointing out the  work of Yamagishi \cite{MR1808626}. This work was supported by the grants DMS-1001740 and DMS-1045119, as well as the ERC-Starting grant ``Nonarcomp" no. 307856.

\section{The formal construction of the curves $C_s$}

In this and the next two sections, we let $f\colon \C^2\to \C^2$ be the polynomial map $f(x,y) = (x^2 - y^4, y^4)$. We will write  $S$ to denote the space of binary sequences $S = \{0,1\}^\N$, and write $\sigma$ to denote the left-shift map $\sigma\colon S\to S$.

We now define a family $\{\varphi_s\}_{s\in S}$ of irreducible formal power series $\varphi_s\in \C\llbracket x,y\rrbracket$ of the  form \begin{equation}\label{eqn1} \varphi_s(x,y) = x + a_0^sy^2 + a_1^sy^6 + \cdots + a_n^sy^{2+4n} + \cdots,\end{equation} by recursively defining the coefficients $a_n^s$, in the following manner. First, we set $a_0^s = (-1)^{s_0}$; then, assuming $a_0^s,\ldots, a_n^s$ have been defined for all $s\in S$, we set\begin{equation}\label{recursion}
a_{n+1}^s = \begin{cases}\displaystyle-\frac{1}{2a_0^s}\sum_{\substack{i+j = n+1\\i,j\geq 1}} a_i^sa_j^s & \mbox{ if }4\nmid n.\bigskip\\
\displaystyle-\frac{a_{n/4}^{\sigma(s)}}{2a_0^s} - \frac{1}{2a_0^s}\sum_{\substack{i+j = n+1\\i,j\geq 1}} a_i^sa_j^s & \mbox{ if }4\mid n.\end{cases}\end{equation} Let $C_s$ denote the formal curve through the origin in $\C^2$ defined by $\varphi_s$. As the next proposition shows, the curves $\{C_s\}_{s\in S}$ have the properties discussed in the introduction.

\begin{prop}\label{formalprop} The formal curves $\{C_s\}_{s\in S}$ satisfy \begin{enumerate}
\item[$1.$] $f(C_s) = C_{\sigma(s)}$ for all $s\in S$, and
\item[$2.$] the local intersection multiplicity $C_s\cdot C_t$ is $\frac{1}{3}(4^{m+1} + 2)$, where $m$ is the smallest index such that $s_m\neq t_m$.
\end{enumerate}
\end{prop}
\begin{proof} To prove 1., we must show that $\varphi_s\mid (\varphi_{\sigma(s)}\circ f)$ in the ring $\C\llbracket x,y\rrbracket$. Indeed, we show that $\varphi_{\sigma(s)}\circ f =  (x + a_0^sy^2 + a_1^sy^6 + \cdots)(x - a_0^sy^2 - a_1^sy^6 - \cdots)$. To see this,  observe that  \begin{equation}\label{eqn3}
(x + a_0^sy^2 + a_1^sy^6 + \cdots)(x - a_0^sy^2 - a_1^sy^6 - \cdots) = x^2 - y^4 - \sum_{n\geq 0}\sum_{\substack{i+j = n+1\\i,j\geq 0}}a_i^sa_j^sy^{4(n+2)}.\end{equation}  The recursion formula (\ref{recursion}) then gives that the coefficient of $y^{4(n+2)}$ in this expression is $0$ when $4\nmid n$ and is $a_{n/4}^{\sigma(s)}$ when $4\mid n$, so the right hand side of (\ref{eqn3}) is \[x^2 - y^4 + \sum_{k\geq 0} a_k^{\sigma(s)}y^{4(4k + 2)} = x^2 - y^4 + \sum_{k\geq 0}a_k^{\sigma(s)}y^{8 + 16k} = \varphi_{\sigma(s)}\circ f.\] This completes the proof of 1. 

To prove 2., we first make the easy observation that the intersection multiplicity $C_s \cdot C_t$ is precisely the smallest integer $k$ such that the coefficients of $y^k$ in the power series $\varphi_s$ and $\varphi_t$ differ. From equation (\ref{eqn1}), it then follows that $C_s\cdot C_t = 2 + 4n$, where $n$ is the smallest integer such that $a_n^s\neq a_n^t$. We will prove 2.\ by induction on $m\geq 0$, where $m$ is the smallest index such that $s_m\neq t_m$. If $m = 0$, then $a_0^s\neq a_0^t$, and hence $C_s\cdot C_t = 2$, establishing the base case of the induction. Now assume that $m>0$ is the smallest index index such that $s_m\neq t_m$. Then, by induction, $C_{\sigma(s)}\cdot C_{\sigma(t)} = \frac{1}{3}(4^m + 2)$, from which it follows that the first index $n$ for which $a_n^{\sigma(s)}\neq a_n^{\sigma(t)}$ is $n = \frac{1}{3}(4^{m-1}-1)$. Using the recursion formula (\ref{recursion}), we can then conclude that the first index $n$ such that $a_n^s\neq a_n^t$ is \[
n = 1 + \frac{4}{3}(4^{m-1} - 1) = \frac{1}{3}(4^m-1).\] Thus $C_s\cdot C_t = 2 + \frac{4}{3}(4^m-1) = \frac{1}{3}(4^{m+1}+2),$ completing the induction, and the proof.
\end{proof}

\section{Analyticity}

In this section, we complete the proof of \hyperref[thmA]{Theorem~\ref*{thmA}} by proving that each of the power series $\varphi_s$ constructed in \S1 are convergent. Indeed, using very crude estimates, we will prove the following proposition.

\begin{prop}\label{analyticity} Let $C = 1/20$ and $R = 10$. Then $|a_n^s|\leq CR^n/n^2$ for each $n\geq 1$ and each $s\in S$. In particular, $\varphi_s$ converges on the set $\{(x,y)\in \C^2 : |y|< 1/10\}$.
\end{prop}

To prove the proposition, we will make use of the following lemma.

\begin{lem}\label{lemma} Let $n\geq 1$ be an integer. Then \[\sum_{k=1}^n\frac{1}{k^2(n-k+1)^2}\leq \frac{20}{(n+1)^2}.\]
\end{lem}
\begin{proof} The symmetry in the terms of the left hand sum implies that \[
\sum_{k=1}^n\frac{1}{k^2(n-k+1)^2}\leq 2\sum_{k=1}^{\lfloor\frac{n+1}{2}\rfloor} \frac{1}{k^2(n-k+1)^2}.\] Multiplying both sides of this inequality by $(n+1)^2$ yields \[\sum_{k=1}^n\frac{(n+1)^2}{k^2(n-k+1)^2}\leq 2\sum_{k=1}^{\lfloor\frac{n+1}{2}\rfloor} \frac{(n+1)^2}{k^2(n-k+1)^2} = 2\sum_{k=1}^{\lfloor\frac{n+1}{2}\rfloor}\frac{1}{k^2(1 - \frac{k}{n+1})^2}\leq 8\sum_{k=1}^{\lfloor\frac{n+1}{2}\rfloor}\frac{1}{k^2}<\frac{8\pi^2}{6}.\] Since $8\pi^2/6<20$, the proof is complete.
\end{proof}

\begin{proof}[Proof of Proposition \ref{analyticity}] We will prove the proposition by induction on $n\geq 1$. When $n = 1$, the recursion formula (\ref{recursion}) gives $a_1^s = -a_0^{\sigma(s)}/2a_0^s =\pm\frac{1}{2}$ for each $s\in S$, and hence $|a_1^s| = \frac{1}{2} = CR$, establishing the base case of the induction. Now assume that the proposition holds for $a_k^s$ when $k\leq n$. If $4\nmid n$, then (\ref{recursion}), the triangle inequality, and \hyperref[lemma]{Lemma~\ref*{lemma}} give that\[
|a_{n+1}^s|\leq \frac{1}{2} \sum_{k=1}^n\frac{C^2R^{n+1}}{k^2(n-k+1)^2}\leq \frac{20C^2R^{n+1}}{2(n+1)^2} = \frac{CR^{n+1}}{2(n+1)^2}<\frac{CR^{n+1}}{(n+1)^2},\] establishing the proposition in this case. If $4\mid n$, then (\ref{recursion}) gives \begin{equation}\label{eqn4}|a_{n+1}^s|\leq \frac{CR^{n/4}}{2(n/4)^2} + \frac{1}{2}\sum_{k=1}^n\frac{C^2R^{n+1}}{k^2(n-k+1)^2}\leq \frac{CR^{n/4}}{2(n/4)^2} + \frac{CR^{n+1}}{2(n+1)^2}.\end{equation}Since $4\mid n$, and in particular $n\geq 4$, the inequality $\frac{n}{4}\leq (n+1) - 4$ is valid, and hence \[\frac{CR^{n/4}}{2(n/4)^2} = \frac{8CR^{n/4}}{n^2}\leq \frac{8CR^{n+1}}{n^2R^4}<\frac{8CR^{n+1}}{(n+1)^2R^4}.\] Putting this estimate into (\ref{eqn4}), we see that \[
|a_{n+1}^s|\leq \left(\frac{8}{R^4} + \frac{1}{2}\right)\frac{CR^{n+1}}{(n+1)^2}<\frac{CR^{n+1}}{(n+1)^2}.\] This completes the proof.
\end{proof}

We have thus shown that the curves $C_s$ are holomorphic. On the other hand, it should be noted that they cannot all be algebraic. This is because if $C$ and $D$ are algebraic plane curves passing through $0$, then, as a simple consequence of Bezout's theorem,  the local intersection multiplicities $C\cdot f^n_*D$ grow subexponentially in $n$. We are then led to the following interesting, but possibly difficult, questions.

\begin{Quest} Which, if any, of the curves $C_s$ just constructed are (local irreducible components of) germs of algebraic curves? More specifically, if a sequence $s\in S$ is not eventually periodic, is it possible for $C_s$ to be a germ of an algebraic curve?
\end{Quest}

\section{Realization as a Cantor bouquet}

In this section, we reconstruct the curves $C_s$ from \S1 as a \emph{Cantor bouquet} of holomorphic stable manifolds using a geometric procedure  given first by Yamagishi in \cite{MR1808626}, see also \cite{Yam2, MR2195140, MR2307152, Tomoko, MR2629648}. We only sketch this construction here; for details see  \cite[\S2]{MR1808626}.

Let $\pi\colon X\to \C^2$ denote the blowup of the origin in $\C^2$, and let $E$ denote the exceptional divisor of $\pi$. It is easy to check that the lift $f_X\colon X\dashrightarrow X$ of $f$ has exactly one indeterminacy point $q$, given by $z = w = 0$ in the local coordinates $z = x/y$ and $w = y$. 

Now let $\pi'\colon X'\to X$ denote the blowup of the point $q$. It is a straightforward computation to check that  $f$ lifts to an (everywhere defined) holomorphic map $F\colon X'\to X$. Moreover, with respect to the local coordinates $z,w$ on $X$ and $u = z/w$, $v = w$ on $X'$, the map $F$ is given simply as $F(u,v) = (u^2 - 1, v^4)$. The preimage $F^{-1}(q)$ then consists of two points, $q_0 = (1,0)$ and $q_1 = (-1,0)$. Because $F$ is not a local biholomorphism near either $q_0$ or $q_1$, we are not in an identical situation to the one considered in \cite[\S2]{MR1808626}, but nonetheless the map $(\pi')^{-1}\circ F$, defined away from the $q_i$, exhibits similar local dynamics around the $q_i$ as the maps studied in \cite[\S2]{MR1808626}; specifically, $(\pi')^{-1}\circ F$ is contracting in the $v$-direction and expanding in the $u$-direction near the $q_i$. It is this behavior that allows us to consider holomorphic stable manifolds of $q$. Let $U$ be a small neighborhood of $q$ in $X$, so that $F^{-1}(U)$ is a disjoint union $U_0\sqcup U_1$ of neighborhoods of $q_0$ and $q_1$, respectively. Let \begin{align*}
B & = \{p\in X\smallsetminus E : f_X^n(p)\in U \mbox{ for all } n\geq 0\}\\
& = \{p\in X\smallsetminus E : f_X^n(p)\in \pi'(U_0\sqcup U_1)\mbox{ for all }n\geq0\}.
\end{align*} Near $q$, the set $B$ will be the union of the (strict transforms in $X$ of the) curves $C_s$ constructed in \S1. To be precise, if $s\in S := \{0,1\}^\N$, then the set $\wt{C}_s := \{p\in X\smallsetminus E : f_X^n(p)\in \pi'(U_{s_n})$ for all $n\geq 0\}$ is a curve transverse to $E$ at $q$, and the family $\{\wt{C}_s\}_{s\in S}$ consists of the strict transforms in $X$ of the curves $C_s$ from \S1. The $\wt{C}_s$ are \emph{local stable manifolds} of $q$ in the sense that they form an invariant family for $f_X$, and $f_X^n(p)\to q$ as $n\to\infty$ for all $p\in B = \bigcup_s \wt{C}_s$ near $q$. Moreover, it is clear from this construction that $f(\wt{C}_s) = \wt{C}_{\sigma(s)}$, where $\sigma\colon S\to S$ is the left shift map.

It is also possible to use the geometry in this construction to compute the local intersection multiplicites of the curves $C_s$ at the origin, rederiving \hyperref[formalprop]{Proposition~\ref*{formalprop}(2)}. To see how, first observe that because the $\wt{C}_s$ are transverse to $E$ at $q$, the projection formula gives \[
C_s\cdot C_t = (\pi^*C_s \cdot \pi^*C_t) = (\wt{C}_s + E)\cdot (\wt{C}_t + E) = \wt{C}_s\cdot \wt{C}_t + 1.\] Similarly, if $D_s$ denotes the strict transform of $C_s$ in $X'$, then $\wt{C}_s\cdot \wt{C}_t = D_s\cdot D_t + 1$, and thus $C_s\cdot C_t = D_s\cdot D_t + 2$. If $s_0\neq t_0$, then the germs $D_s$ and $D_t$ lie in different open sets $U_0$ and $U_1$, and hence do not intersect, proving that $C_s\cdot C_t = 2$, as previously derived. Suppose, on the other hand, that $s_0 = t_0$, say without loss of generality $s_0 = t_0 = 0$. If $F_0$ denotes the restriction $F_0 = F|_{U_0}$, then $D_s = F_0^*\wt{C}_{\sigma(s)}$ and $D_t = F_0^*\wt{C}_{\sigma(t)}$. Because $F_0$ has local topological degree $4$ at $q_0$, it follows that \begin{align*}
C_s\cdot C_t & = 2 + D_s\cdot D_t = 2 + F_0^*\wt{C}_{\sigma(s)}\cdot F_0^*\wt{C}_{\sigma(t)} = 2 + 4(\wt{C}_{\sigma(s)}\cdot \wt{C}_{\sigma(t)})\\
& = 4(C_{\sigma(s)}\cdot C_{\sigma(t)}) - 2\end{align*} when $s_0= t_0$. Using this identity and the fact that $C_s\cdot C_t = 2$ when $s_0\neq t_0$, one easily rederives \hyperref[formalprop]{Proposition~\ref*{formalprop}(2)}.

Finally, it is worth pointing out that this  argument applies equally well to any other superattracting germ with similar geometry. For instance, using either the methods in this section or in \S1, one can show that the maps $f(x,y) = (x^p - y^q, y^r)$ where $2\leq p < r\leq q$ all have Cantor bouquets of curves.

\section{Mixed multiplicities and the proof of Theorem B}

In this section, we give a proof of \hyperref[thmB]{Theorem~\ref*{thmB}} using Teissier's theory of \emph{mixed multiplicities}. In fact, we will prove a slightly stronger theorem, namely \hyperref[thmC]{Theorem~\ref*{thmC}} below. For ease of notation, let $R$ denote the formal power series ring $\C\llbracket x_1,\ldots, x_d\rrbracket$, where  $d\geq 2$ is a fixed integer. Recall that $R$ is a local ring with maximal ideal $\mf{m} = (x_1,\ldots, x_d)$. An ideal $\mf{a}$ of $R$ is said to be $\mf{m}$-primary if $\mf{a}$ contains some power of $\mf{m}$, or equivalently if $\mf{a}$ defines the origin $0\in \C^d$.

\begin{thm}\label{thmC} Let $f\colon (\C^d,0)\to (\C^d,0)$ be a holomorphic fixed point germ at the origin $0\in \C^d$ that is finite-to-one near $0$. Let $\mf{a}_1,\ldots, \mf{a}_d$ be $\mf{m}$-primary ideals of $R$. Choose generators $\psi_1^{(i)},\ldots, \psi_{m_i}^{(i)}$ for each of the ideals $\mf{a}_i$. For every point $z\in \C^{m_i}$, let $D_{z}^{(i)}$ denote the formal hypersurface germ $D_{z}^{(i)} = \{z_{1}\psi_1^{(i)} + \cdots + z_{m_i}\psi_{m_i}^{(i)} = 0\}$ through the origin. Fix an integer $k\in \{1,\ldots, d-1\}$. For each point $a = (a_1,\ldots, a_k)\in \C^{m_1}\times\cdots\times \C^{m_k}$ and $b = (b_1,\ldots, b_{d-k})\in \C^{m_{k+1}}\times\cdots\times \C^{m_d}$, let $V_a$ and $W_b$ be the local intersection cycles $V_a := D_{a_1}^{(1)}\cdot\ldots\cdot D_{a_k}^{(k)}$ and $W_b := D_{b_1}^{(k+1)}\cdot\ldots \cdot D_{b_{d-k}}^{(d)}$. Then there is a dense subset $U\subseteq (\C^{m_1}\times\cdots\times \C^{m_k})\times(\C^{m_{k+1}}\times\cdots\times \C^{m_d})$, given as the complement of a countable union of algebraic subsets, such that for all $(a,b)\in U$, the cycles $V_a$ and $W_b$ are of codimension $k$ and $d - k$, respectively, and the sequence of local intersection multiplicities $\mu(n) = V_a\cdot f^n_*W_b$ grows subexponentially.
\end{thm}

Observe that \hyperref[thmC]{Theorem~\ref*{thmC}} implies \hyperref[thmB]{Theorem~\ref*{thmB}} by simply taking each of the ideals $\mf{a}_i$ of \hyperref[thmC]{Theorem~\ref*{thmC}} to be the same ideal $(\psi_1,\ldots, \psi_m)$.

Before beginning the proof of \hyperref[thmC]{Theorem~\ref*{thmC}}, we recall some basic facts from the theory of mixed multiplicities developed by B.\ Teissier in the 70s \cite{MR0374482, MR0467800, MR645731, MR708342, MR518229}. A concise and clear overview to the topic that suffices for our purposes can be found in \cite[\S1.6.8]{MR2095471}. For us, the relevant results are the following.

\begin{thm}[Teissier]\label{teissier1} Let $\mf{b}_1,\ldots, \mf{b}_d$ be $\mf{m}$-primary ideals of $R$. Fix generators $\varphi_1^{(i)},\ldots, \varphi_{m_i}^{(i)}$ of each of the ideals $\mf{b}_i$. For a point $z\in \C^{m_i}$, let $D_{z}^{(i)} = \{z_{1}\varphi_1^{(i)} + \cdots + z_{m_i}\varphi_{m_i}^{(i)} = 0\}$. Then there is an integer $e(\mf{b}_1;\cdots; \mf{b}_d)\geq 1$ and a nonempty Zariski open subset $U\subseteq \C^{m_1}\times\cdots\times \C^{m_d}$ such that if $ (a_1,\ldots, a_d)\in U$, then the hypersurface germs $D_{a_1}^{(1)},\ldots, D_{a_d}^{(d)}$ intersect properly at the origin, and $e(\mf{b}_1; \cdots; \mf{b}_d)$ is exactly the local intersection multiplicity $ D_{a_1}^{(1)}\cdot D_{a_2}^{(2)}\cdot\ldots\cdot D_{a_d}^{(d)}$. Moreover, one has the inequality \[
e(\mf{b}_1; \cdots; \mf{b}_d) \leq e(\mf{b}_1)^{1/d}\cdots e(\mf{b}_d)^{1/d},\] where here $e(\mf{b}_i)$ denotes the standard Samuel multiplicity of $\mf{b}_i$, that is \[
e(\mf{b}_i) := \lim_{n\to \infty} \frac{d!}{n^d}\,\mathrm{length}_R(R/\mf{a}^{n+1}) \in \N.\] The integer $e(\mf{b}_1;\cdots; \mf{b}_d)$ is called the \emph{mixed multiplicity} of the ideals $\mf{b}_i$.
\end{thm}

With these facts at our disposal, we can now prove \hyperref[thmC]{Theorem~\ref*{thmC}}.

\begin{proof}[Proof of {\hyperref[thmC]{Theorem~\ref*{thmC}}}] The projection formula says precisely that \begin{equation}\label{proj}
\mu(n) = V_a\cdot f^n_*W_b = f^{n*}D_{a_1}^{(1)}\cdot \ldots \cdot f^{n*}D_{a_k}^{(k)}\cdot D_{b_1}^{(k+1)}\cdot\ldots\cdot D_{b_{d-k}}^{(d)}.\end{equation} By \hyperref[teissier1]{Theorem~\ref*{teissier1}}, for each  $n$ there is a nonempty Zariski open subset $U_n\subseteq (\C^{m_1}\times\cdots\times \C^{m_k})\times(\C^{m_{k+1}}\times\cdots\times \C^{m_d})$ such that if $(a,b)\in U_n$, then the right hand side of equation (\ref{proj}) is exactly the mixed multiplicity $e(f^{n*}\mf{a}_1; \cdots; f^{n*}\mf{a}_k; \mf{a}_{k+1}; \cdots; \mf{a}_d)$, where $f^{n*}\mf{a}_i$ is the ideal $(\psi_1^{(i)}\circ f^n, \ldots, \psi_{m_i}^{(i)}\circ f^n)$. We point out that $f^{n*}\mf{a}_i$ is an $\mf{m}$-primary ideal by our assumption that $f$ is finite-to-one near $0$. Let $U = \bigcap_n U_n$. For $(a,b)\in U$, this proves that $\mu(n) = e(f^{n*}\mf{a}_1; \cdots; f^{n*}\mf{a}_k; \mf{a}_{k+1}; \cdots; \mf{a}_d)$ for all $n$.

The problem is now reduced to showing that the sequence $e(f^{n*}\mf{a}_1; \cdots; f^{n*}\mf{a}_k; \mf{a}_{k+1}; \cdots; \mf{a}_d)$ grows subexponentially. Again using \hyperref[teissier1]{Theorem~\ref*{teissier1}}, we see \[
e(f^{n*}\mf{a}_1; \cdots; f^{n*}\mf{a}_k; \mf{a}_{k+1}; \cdots; \mf{a}_d)\leq e(f^{n*}\mf{a}_1)^{1/d}\cdots e(f^{n*}\mf{a}_k)^{1/d}e(\mf{a}_{k+1})^{1/d}\cdots e(\mf{a}_d)^{1/d},\] so it suffices to show that $e(f^{n*}\mf{a}_i)$ grows subexponentially. Let $r\geq 1$ be an integer such that $\mf{m}^r\subseteq \mf{a}_i$ for each $i$, and let $s\geq 1$ be an integer such that $\mf{m}^s\subseteq f^*\mf{m}$. Then one has inclusions\[
f^{n*}\mf{a}_i\supseteq f^{n*}\mf{m}^r\supseteq f^{(n-1)*}\mf{m}^{sr}\supseteq f^{(n-2)*}\mf{m}^{s^2r}\supseteq\cdots\supseteq \mf{m}^{s^nr}.\] It follows that \[
e(f^{n*}\mf{a}_i) \leq e(\mf{m}^{s^nr}) = (s^nr)^d e(\mf{m}) = r^ds^{dn},\] which grows subexponentially, completing the proof. 
\end{proof}

\begin{rem} With very little extra work, one can show that the exponential growth rate of the sequence $e(f^{n*}\mf{a}_i)^{1/d}$ can be bounded above in the following manner:\[ 
\lim_{n\to \infty} \frac{1}{n} \log e(f^{n*}\mf{a}_i)^{1/d} \leq \lim_{n\to \infty}\frac{1}{n} \log\min\{s\geq 1 : \mf{m}^s\subseteq f^{n*}\mf{m}\}.\] We mention this because quantities such as that on the right hand side have been recently studied by Majidi-Zolbanin, Miasnikov, and Szpiro \cite{MZMS}. In their notation, the right hand side of this inequality is exactly $w_h(f)$.
\end{rem}

\section{Valuative dynamics and the proof of Theorem C}

In this final section, we will prove \hyperref[thmD]{Theorem~\ref*{thmD}} using techniques from valuation theory and dynamics on valuation spaces. Let us begin by recalling the setup of the theorem. We fix a superattracting holomorphic fixed point germ $f\colon (\C^2,0)\to (\C^2,0)$, which we assume to be finite-to-one near $0$. Let $\mf{a} = (\psi_1,\ldots, \psi_m)$ be an $\mf{m}$-primary ideal in the formal power series ring $\C\llbracket x,y\rrbracket$. For any $z\in \C^m$, we set $D_z = \{z_1\psi_1 + \cdots + z_m\psi_m = 0\}$. We aim to show that there is a dense subset $U\subseteq \C^m\times \C^m$, given as the complement of a countable union of algebraic subsets, such that for all $(z,w)\in U$, the sequence $\mu(n) := D_z\cdot f^n_*D_w$ eventually satisfies an integral linear recursion relation. We have already seen in \S4 that we can find such a set $U$ for which one has $\mu(n) = e(f^{n*}\mf{a}; \mf{a})$ for all $(z,w)\in U$ and all $n\geq 1$. Our starting point in this section is the following alternate characterization of mixed multiplicities, which can be found in \cite[\S1.6.8]{MR2095471} and \cite{MR0354663}.

\begin{thm}\label{teissier2} Let $\mf{b}_1$ and $\mf{b}_2$ be two $\mf{m}$-primary ideals of $\C\llbracket x,y\rrbracket$. Let $\pi\colon X\to (\C^2,0)$ be a modification over $0$ which dominates the normalized blowup of each of the ideals $\mf{b}_i$. Then there exist divisors $Z_1$ and $Z_2$ of $X$, both supported within the exceptional locus $\pi^{-1}(0)$ of $\pi$, for which $\mf{b}_i\cdot \mc{O}_X = \mc{O}_X(Z_i)$. The mixed multiplicity $e(\mf{b}_1; \mf{b}_2)$ is given by the intersection number $-(Z_1\cdot Z_2)$. Finally, the divisors $Z_i$ are \emph{relatively nef}, which is to say that for every irreducible component $E$ of $\pi^{-1}(0)$, one has $Z_i\cdot E \geq 0$.
\end{thm}

Here, and for the rest of the article, a \emph{modification} $\pi\colon X\to (\C^2,0)$ over $0$ is defined to be a proper birational morphism $\pi\colon X\to \C^2$ from a normal variety $X$ to $\C^2$ that is an isomorphism over $\C^2\smallsetminus \{0\}$. Such a modification will be called a \emph{blowup}  if it is obtained as a composition of point blowups. If $\pi\colon X\to (\C^2,0)$ is any modification, and $E_1,\ldots, E_r$ denote the irreducible components of the exceptional locus $\pi^{-1}(0)$, we define $\Div(\pi)$ to be the vector space of $\R$-divisors $\Div(\pi) = \bigoplus_{i=1}^r \R E_i$. The intersection pairing on $\Div(\pi)$ is nondegenerate by  the Hodge index theorem  \cite[Theorem V.1.9]{MR0463157}, and thus there is a dual basis $\check{E}_1,\ldots, \check{E}_r\in \Div(\pi)$, i.e., a basis satisfying the relation $\check{E}_i\cdot E_j = \delta_{ij}$.

Before beginning the proof of \hyperref[thmD]{Theorem~\ref*{thmD}}, let us give a (very) brief overview of the \emph{valuative tree} $\mc{V}$ at the origin $0\in \C^2$ of Favre-Jonsson, which is the relevant valuation space for us. A full-on introduction to the valuative tree would take us too far afield here; detailed references can be found in \cite{MR2097722, Mattias}, and more concise introductions can be found in \cite{MR2339287, GR}. 


The valuative tree $\mc{V}$ at $0\in \C^2$ is defined as the set of all semivaluations $\nu\colon \C\llbracket x,y\rrbracket\to \R\cup\{+\infty\}$ with the properties that $\nu|_{\C^\times} \equiv 0$ and $\min\{\nu(x), \nu(y)\} = 1$. For us, the most important example is that of a \emph{divisorial valuation}. A valuation $\nu\in \mc{V}$ is divisorial if there is a blowup $\pi\colon X\to (\C^2,0)$, an irreducible component $E$ of the exceptional locus $\pi^{-1}(0)$, and a constant $\lambda\in \R$ such that $\nu(P) = \lambda\ord_E(P\circ \pi)$ for all $P\in \C\llbracket x,y\rrbracket$. In this case, the constant $\lambda$ is exactly $\lambda = b_E^{-1}$, where $b_E = \min\{\ord_E(x\circ \pi), \ord_E(y\circ \pi)\}\in \N$. The constant $b_E$ is sometimes called the \emph{generic multiplicity} of $E$. If $\nu$ is a divisorial valuation of the above form, we will denote it simply as $\nu_E$.

Suppose that $\pi\colon X\to (\C^2,0)$ is a blowup, and $E_1,\ldots, E_r$ are the irreducible components of $\pi^{-1}(0)$. Then any divisorial valuation $\nu\in \mc{V}$ defines a linear functional $\nu\colon\Div(\pi)\to \R$; essentially, $\nu(E_i)$ is the $\nu$-valuation of a local defining equation of $E_i$ at the \emph{center} of $\nu$ in $\pi$, see \cite[\S1.2]{favre:holoselfmapssingratsurf} for details. Since the intersection pairing is nondegenerate, it follows that there is a divisor $Z_{\nu, \pi}\in \Div(\pi)$ such that $\nu(D) = Z_{\nu, \pi}\cdot D$ for all $D\in \Div(\pi)$. If $\pi'\colon X'\to (\C^2, 0)$ is a blowup dominating $\pi$, say $\eta\colon X'\to X$ is such that $\pi' = \pi\eta$, then $Z_{\nu, \pi'} = \eta^*Z_{\nu, \pi}$ and $Z_{\nu, \pi} = \eta_*Z_{\nu, \pi'}$. Finally, if $\nu = b_{E_i}^{-1}\ord_{E_i}$ for one of the exceptional components $E_i$ of $\pi^{-1}(0)$, then one easily checks that $Z_{\nu, \pi} = b_{E_i}^{-1}\check{E}_i$.

The valuative tree $\mc{V}$ has a natural topology and a natural poset structure $(\mc{V}, \leq)$. With respect to these structures, $\mc{V}$ is a \emph{rooted tree} (see \cite[\S2]{Mattias} for a precise definition). For any two elements $\nu_1,\nu_2\in \mc{V}$, there is a unique greatest element $\nu_1\wedge \nu_2$ that is both $\leq \nu_1$ and $\leq \nu_2$. In addition, there is defined on $\mc{V}$ an increasing function $\alpha\colon \mc{V}\to [1,+\infty]$, called the \emph{skewness} function, which is finite on divisorial valuations and has the following geometric property: if $\pi\colon X\to (\C^2,0)$ is a blowup and $E_1$ and $E_2$ are two irreducible components of the exceptional locus $\pi^{-1}(0)$, then 
\begin{equation}\label{Beqn3}\alpha(\nu_{E_1}\wedge \nu_{E_2}) = -(Z_{\nu_{E_1}, \pi}\cdot Z_{\nu_{E_2}, \pi}).\end{equation}

Finally, we note that $f$ induces in a natural way a dynamical system $f_\bullet \colon \mc{V}\to \mc{V}$. Indeed, if $\nu\in \mc{V}$, then we obtain a semivaluation $f_*\nu$ defined by $(f_*\nu)(P) = \nu(P\circ f)$. In general the value $c(f,\nu):= \min\{(f_*\nu)(x), (f_*\nu)(y)\}$ is greater than $1$, so $f_*\nu$ is not an element of $\mc{V}$, but by appropriately normalizing we obtain a semivaluation $f_\bullet\nu = c(f,\nu)^{-1}f_*\nu\in \mc{V}$. The quantity $c(f,\nu)$ is called the \emph{attraction rate} of $f$ along $\nu$, and is the primary object of study in the paper \cite{GR}, the results of which we will use shortly.

\begin{proof}[Proof of {\hyperref[thmD]{Theorem~\ref*{thmD}}}] We begin the proof by deriving an alternate expression for the mixed multiplicity $e(f^*\mf{a}; \mf{a})$ using the valuative language just discussed. Let $\pi_1\colon X_1\to (\C^2,0)$ and $\pi_2\colon X_2\to (\C^2,0)$ be the normalized blowups of the ideals $f^*\mf{a}$ and $\mf{a}$, respectively, and let $Z_i$ be the divisors on $X_i$ for $i = 1,2$ such that $f^*\mf{a}\cdot \mc{O}_{X_1} = \mc{O}_{X_1}(Z_1)$ and $\mf{a}\cdot \mc{O}_{X_2} = \mc{O}_{X_2}(Z_2)$. Let $E_1,\ldots, E_r$ be the irreducible components of the exceptional locus $\pi_2^{-1}(0)$ of $\pi_2$. For each $i = 1,\ldots, r$, let $a_i\in \Z$ be the integer $Z_2\cdot E_i$, so that $Z_2$ can be written $Z_2 = \sum_{i=1}^r a_i\check{E_i}$. By \hyperref[teissier2]{Theorem~\ref*{teissier2}}, these integers $a_i$ are nonnegative.

As a consequence of Hironaka's theorem on resolution of singularities, it is possible to find blowups $\eta_1\colon Y_1\to (\C^2,0)$ and $\eta_2\colon Y_2\to (\C^2,0)$ over $0$ with the following properties:\begin{enumerate}
\item[1.] Each $\eta_i$ is a \emph{log resolution} of both of the ideals $\mf{a}$ and $f^*\mf{a}$, that is to say, the ideals $\mf{a}\cdot \mc{O}_{Y_i}$ and $f^*\mf{a}\cdot \mc{O}_{Y_i}$ are locally principal. In particular, the $\eta_i$ dominate both $\pi_1$ and $\pi_2$, so there exist proper birational morphisms $\sigma_i\colon Y_1\to X_i$ and $\gamma_i\colon Y_2\to X_i$ for $i = 1,2$ such that one has $\eta_1 = \pi_i\sigma_i$ and $\eta_2 = \pi_i\gamma_i$.
\item[2.]  The map $f$ lifts to a holomorphic map $F \colon Y_1\to Y_2$, or in other words, $F = \eta_2^{-1}f\eta_1$ has no indeterminacy points.
\item[3.] If $\wt{E}_1,\ldots, \wt{E}_r$ denote the strict transforms of the $E_i$ in $Y_1$ under $\sigma_2$, then $F$ does not contract any of the $\wt{E}_i$ to a point.
\end{enumerate} The ideal $f^*\mf{a}\cdot \mc{O}_{Y_1}$ is obtained on the one hand by first pulling back the ideal $\mf{a}$ by $f$ to get $f^*\mf{a}$, and then by pulling back $f^*\mf{a}$ by $\eta_1$ to get $f^*\mf{a}\cdot \mc{O}_{Y_1}$. On the other hand, because $f\eta_1 = \eta_2F$, we may also obtain $f^*\mf{a}\cdot \mc{O}_{Y_1}$ by first pulling back $\mf{a}$ by $\eta_2$ to get $\mf{a}\cdot \mc{O}_{Y_2} = \mc{O}_{Y_2}(\gamma_2^*Z_2)$, and then pulling this back by $F$ to get $f^*\mf{a}\cdot \mc{O}_{Y_1} = \mc{O}_{Y_1}(F^*\gamma_2^*Z_2)$. Using this and the fact that $\mf{a}\cdot \mc{O}_{Y_1} = \mc{O}_{Y_1}(\sigma_2^*Z_2)$, \hyperref[teissier2]{Theorem~\ref*{teissier2}} implies that \[
e(f^*\mf{a};\mf{a}) = -( \sigma_2^*Z_2\cdot F^*\gamma_2^*Z_2) = -(F_*\sigma_2^*Z_2\cdot \gamma_2^*Z_2).\] Using our  previously derived expression $Z_2 = \sum_{i=1}^r a_i\check{E}_i$, we can express this as \begin{equation}\label{Beqn1}
e(f^*\mf{a}; \mf{a}) = -\sum_{i,j = 1}^r a_ia_j(F_*\sigma_2^*\check{E}_i \cdot \gamma_2^*\check{E}_j) = -\sum_{i,j=1}^r a_ia_jb_{E_i}b_{E_j}(F_*Z_{\nu_{E_i},\eta_1} \cdot Z_{\nu_{E_j},\eta_2}).\end{equation} Because of our assumption that $F$ does not contract any of the $\wt{E}_i$ to a point, we may now apply \cite[Lemma 1.10]{favre:holoselfmapssingratsurf} to conclude that \begin{align*}
e(f^*\mf{a}; \mf{a}) & = -\sum_{i,j=1}^r a_ia_jb_{E_i}b_{E_i}c(f, \nu_{E_i})(Z_{f_\bullet\nu_{E_i}, \eta_2}\cdot Z_{\nu_{E_j}, \eta_2})\\
& = \sum_{i,j =1}^r a_ia_jb_{E_j}b_{E_j}\alpha(f_\bullet\nu_{E_i}\wedge \nu_{E_j})c(f, \nu_{E_i}).
\end{align*} Of course, this identity is equally valid for any iterate of $f$, leading us to our final equation for the local intersection multiplicities $\mu(n)$. Namely, if $(z,w)\in U$, then \begin{equation}\label{Beqn4}
\mu(n) = \sum_{i,j = 1}^r a_ia_jb_{E_i}b_{E_j}\alpha(f_\bullet^n\nu_{E_i}\wedge \nu_{E_j})c(f^n, \nu_{E_i})\,\,\,\,\,\,\mbox{ for all $n\geq 1$}.\end{equation}

To prove that $\mu(n)$ eventually satisfies an integral linear recursion relation, it therefore suffices to show that for each $i$ and $j$, the sequence $\alpha(f^n_\bullet\nu_{E_i}\wedge \nu_{E_j})c(f^n, \nu_{E_i})$ eventually satisfies an integral linear recursion relation. More generally, we will prove the following: if $\nu\in \mc{V}$ is any divisorial valuation, then $\alpha(f^n_\bullet\nu\wedge \nu_{E_j})c(f^n, \nu)$ eventually satisfies an integral linear recursion relation. 

We first observe that we may without loss of generality prove this for any iterate $f^p$ of $f$. Indeed, the sequence $\alpha(f^n_\bullet\nu\wedge \nu_{E_j})c(f^n,\nu)$ is obtained by joining the sequences \[
\{\alpha((f^{np}_\bullet f^k_\bullet\nu)\wedge \nu_{E_j})c(f^{np}, f^k_\bullet\nu)\}_{n=1}^\infty\,\,\,\,\,\,\,\,\,\,k = 0,1,\ldots, p-1\] in alternating fashion. If each of these sequences eventually satisfies an integral linear recursion relation, then so does the combined sequence.

Immediately let us replace $f$ by $f^2$. By doing so, we may apply \cite[Theorem 3.1]{GR} to conclude that there is a fixed point $\nu_\star\in \mc{V}$ for $f_\bullet$ such that $f^n_\bullet\nu\to \nu_\star$ in a strong sense as $n\to \infty$. In particular, $\alpha(f^n_\bullet\nu\wedge \nu_{E_j})\to \alpha(\nu_\star\wedge \nu_{E_j})<+\infty$. If the sequence $\alpha(f^n_\bullet\nu\wedge \nu_{E_j})$ is eventually constant, then we are done by \cite[Theorem 6.1]{GR}, which says that $c(f^n,\nu)$ eventually satisfies an integral linear recursion relation.

We may assume, therefore, that the sequence $\alpha(f^n_\bullet\nu\wedge \nu_{E_j})$ is not eventually constant; this implies, in particular, that $f^n_\bullet \nu \leq \nu_{E_j}$ for infinitely many $n$. Such a condition imposes strong restrictions on the possible asymptotic behavior of the sequence $f^n_\bullet\nu$. Indeed, the work done in \cite{GR} shows that in this case, $\nu_\star\leq \nu_{E_j}$, and $f^n_\bullet\nu\to \nu_\star$ along a periodic cycle of tangent directions $\vec{v}_1,\ldots, \vec{v}_p$ at $\nu_\star$  (see \cite[\S2]{Mattias} for the notion of tangent directions at a valuation). Replacing $f$ by $f^p$, we can assume without loss of generality that $f^n_\bullet \nu\to \nu_\star$ along a fixed tangent direction $\vec{v}$.

In this situation, it is a rather non-trivial fact (see \cite[\S5.2]{MR2339287}) that one can find a blowup $\pi\colon X\to (\C^2,0)$  and irreducible components $V, W$ of $\pi^{-1}(0)$ which intersect transversely such that the following hold: \begin{enumerate}
\item[1.] The interval $I:= [\nu_{V}, \nu_{W}]\subset \mc{V}$ is invariant for $f_\bullet$, that is, $f_\bullet (I)\subseteq I$.
\item[2.] The interval $I$ contains $\nu_\star$ and intersects the tangent direction $\vec{v}$.
\item[3.] The divisorial valuation $\nu_{E_j}$ corresponds to an irreducible component of $\pi^{-1}(0)$.
\end{enumerate} We now proceed similarly to the proof of \cite[Lemma 6.2]{GR}. For any  valuation $\lambda\in I$ and any $n\geq 1$, one has that $Z_{f^{np}_* \lambda, \pi} = r_n\check{V} + s_n\check{W}$ for some constants $r_n, s_n\geq 1$, and so \[
c(f^{np}, \lambda)\alpha(f^{np}_\bullet \lambda\wedge \nu_{E_j}) = -(Z_{f^{np}_*\lambda, \pi}\cdot Z_{\nu_{E_j}, \pi}) = -r_n(\check{V}\cdot Z_{\nu_{E_j}, \pi}) - s_n(\check{W}\cdot Z_{\nu_{E_j}, \pi}).\] Just as in the proof of \cite[Lemma 6.2]{GR}, there is a $2\times 2$ integer matrix $M$ for which  one has the identity $(r_n, s_n) = (r_{n-1}, s_{n-1})M$. We conclude that $c(f^{np}, \lambda)\alpha(f^{np}_\bullet\lambda\wedge \nu_{E_j})$ satisfies the integral linear recursion relation with characteristic polynomial $t^2 - \mathrm{tr}(M)t + \det(M)$. If $N\geq 1$ is large enough that $f^N_\bullet\nu\in I$, this proves that the sequence $\{(\alpha(f^n_\bullet\nu\wedge \nu_{E_j})c(f^n, \nu)\}_{n = N}^\infty$ satisfies an integral linear recursion relation, completing the proof.
\end{proof}

\begin{rem} One consequence of the study of the sequences $c(f^n, \nu)$ in \cite{GR} is that for any divisorial valuation $\nu$, there is a constant $B = B(\nu)$ such that $c(f^n,\nu)\sim Bc_\infty^n$, where $c_\infty>1$ is the \emph{asymptotic attraction rate} of $f$, that is, \[
c_\infty := \lim_{n\to \infty}\left(\max\{s\geq 1 : f^{n*}\mf{m}\subseteq \mf{m}^s\}\right)^{1/n}.\] Since one has $\alpha(f^n_\bullet\nu_{E_i}\wedge \nu_{E_j}) \leq \alpha(\nu_{E_j})<+\infty$ for all $n$, we can conclude from equation (\ref{Beqn4}) that there exist constants $A_1, A_2>0$ such that $A_1c_\infty^n\leq \mu(n)\leq A_2c_\infty^n$ for all $n$. Moreover, as we saw in the proof of \hyperref[thmD]{Theorem~\ref*{thmD}}, if we replace $f$ by $f^2$, then $f^n_\bullet \nu_{E_i}\to \nu_\star$ for some $\nu_\star\in \mc{V}$, and thus $\alpha(f^n_\bullet\nu_{E_i}\wedge \nu_{E_j})\to \alpha(\nu_\star\wedge\nu_{E_j})<+\infty$. Therefore in this case equation (\ref{Beqn4}) implies that $\mu(n)\sim Ac_\infty^n$ for some constant $A>0$.
\end{rem}

\begin{rem} For any irreducible curve germ $C$ through the origin in $\C^2$, there is a corresponding \emph{curve valuation} $\nu_C\in \mc{V}$, and  one has $f_\bullet \nu_C = \nu_{f(C)}$. Thus the dynamics of $f$ on curves $C$ through the origin is reflected in the dynamics of $f_\bullet$ on curve valuations $\nu_C$. With this in mind, it should not come as a surprise that one can study Arnold's conjecture by examining the dynamics of $f_\bullet$ on $\mc{V}$. In brief, the outline of our proof of \hyperref[thmD]{Theorem~\ref*{thmD}} can be expressed as follows: for ``general" enough curves $C$, the dynamics of $f_\bullet$ on the curve valuations $\nu_C$ is reflected in the dynamics of $f_\bullet$ on certain associated divisorial valuations. The dynamics of $f_\bullet$ on divisorial valuations is very regular: \cite[Theorem 3.1]{GR} says that there is a set $K$ of fixed valuations of $f_\bullet$ that attract all $\nu\in \mc{V}$ with $\alpha(\nu)<+\infty$, which includes all divisorial valuations.

To emphasize this point further, when $f(x,y) = (x^2-y^4, y^4)$ is the example studied in \S\S1-3, the attracting set $K\subset \mc{V}$ consists of a single point $\nu_\star\in \mc{V}$, and the points $\nu\in \mc{V}$ that are \emph{not} attracted to $\nu_\star$ are exactly the curve valuations $\nu_{C_s}$ associated to the curves $C_s$ constructed in \S1. That is, the curve valuations $\nu_{C_s}$ are exactly the points of $\mc{V}$ where the dynamics of $f_\bullet$ is \emph{not} regular.

In the case when $f\colon (\C^2,0)\to (\C^2,0)$ is a finite germ such that the derivative $f'(0)$ has exactly one nonzero eigenvalue, M.\ Ruggiero \cite{MR2904007} has studied the dynamics of $f_\bullet\colon \mc{V}\to \mc{V}$, and found that there exist fixed curve valuations $\nu_1,\nu_2\in \mc{V}$ such that $f^n_\bullet\nu\to \nu_1$ as $n\to \infty$ for all $\nu\in \mc{V}\smallsetminus\{\nu_2\}$. From this one can conclude that the local intersection multiplicities $\mu(n) = C\cdot f^n(D)$ grow at  most exponentially fast for all curve germs $C,D$ through $0$, provided these numbers are always finite, confirming \cite[Theorem 4]{MR1215971}.
\end{rem}

\bibliographystyle{alpha}
\bibliography{References}

\end{document}